\documentclass[12pt]{amsart}
\usepackage[T1]{fontenc}
\usepackage{amsmath,amssymb,amsfonts,amsthm,amsopn}
\usepackage{graphics}
\usepackage{xcolor}
\usepackage{pb-diagram}
\usepackage[title]{appendix}

\newcommand{\tf}{time-frequency}

\newcommand{\tfs}{time-frequency shift}

\newtheorem{theorem}{Theorem}[section]
\newtheorem{lemma}[theorem]{Lemma}

\newtheorem{proposition}[theorem]{Proposition}
\newtheorem{definition}[theorem]{Definition}
\newtheorem{cor}[theorem]{Corollary}

\newtheorem{example}[theorem]{Example}
\newtheorem{remark}[theorem]{Remark}

\newcommand{\beqa}{\begin{eqnarray*}}
	\newcommand{\eeqa}{\end{eqnarray*}}

\newcommand{\field}[1]{\mathbb{#1}}
\newcommand{\bR}{\field{R}}        
\newcommand{\bC}{\field{C}}        
\def\la{\lambda}

\def\eps{\epsilon}

\def\cF{\mathcal{F}}              
\def\cS{\mathcal{S}}
\def\cD{\mathcal{D}}
\def\cH{\mathcal{H}}

\def\cA{\mathcal{A}}

\def\cC{\mathcal{C}}

\def\rd{\bR^d}

\def\rdd{{\bR^{2d}}}

\def\lrd{L^2(\rd)}

\def\<{\left<}
\def\>{\right>}

\def\mv1{M_v^1}

\def\mn{(m,n)}
\def\mn'{(m',n')}

\newcommand{\norm}[1]{\lVert#1\rVert}

\def\i{\infty}

\def\Ren{\mathbb{R}^d}

\def\sch{\mathcal{S}}

\def\f{\varphi}

\def\Sn2{S_{2}(L^{2}(\Ren))}
\def\S1{S_{1}(L^{2}(\Ren))}
\def\sig00{\sigma_{0,0}}

\def\la{\langle}
\def\ra{\rangle}

\newcommand{\A}{\mathcal{A}}

\begin{document}
	
\begin{abstract}
	Time-frequency localization operators, originally introduced by Daubechies (1988), provide a framework for localizing signals in the phase space and have become a central tool in time-frequency analysis. 
	In this paper we introduce and study a broad generalization of these operators, called \emph{$\mathcal{A}$-localization operators}, associated with a metaplectic Wigner distribution $W_\mathcal{A}$ and the corresponding $\mathcal{A}$-pseudodifferential calculus.
	
	We first show that the classical relation between localization operators and Weyl quantization extends to any \emph{covariant metaplectic Wigner distribution}. 
	Specifically, if $W_\mathcal{A}$ satisfies the covariance property
	\[
	W_\mathcal{A}(\pi(z)f,\pi(z)g)=T_zW_\mathcal{A}(f,g), \qquad z\in\mathbb{R}^{2d},
	\]
	then
	\[
	A_{a}^{\varphi_1,\varphi_2}
	= \operatorname{Op}_\mathcal{A}\big(a * W_\mathcal{A}(\varphi_2,\varphi_1)\big),
	\]
	and conversely, this identity characterizes covariance. 
	This result extends the recent representation formula of Bastianoni and Teofanov for $\tau$-operators to the full metaplectic framework.
	
	We then define the $\mathcal{A}$-localization operator $A_{a,\mathcal{A}}^{\varphi_1,\varphi_2}$ and investigate its analytical properties. 
	We establish boundedness results on modulation spaces and provide sufficient conditions for Schatten-von Neumann class membership.
	These findings connect the structure of metaplectic representations with time-frequency localization theory, offering a unified approach to quantization and signal analysis.
\end{abstract}

\title{$\cA$-localization operators}
\author{Elena Cordero}
\address{Universit\`a di Torino, Dipartimento di Matematica, via Carlo Alberto 10, 10123 Torino, Italy}
\email{elena.cordero@unito.it}
\author{Edoardo Pucci}
\address{Universit\`a di Torino, Dipartimento di Matematica, via Carlo Alberto 10, 10123 Torino, Italy}
\email{edoardo.pucci@unito.it}
\thanks{}

\maketitle

\section{Introduction}

Time-frequency localization operators were introduced by Daubechies in 1988~\cite{Daubechies} as a class of operators designed to localize a signal in the phase (or time-frequency) space. 
Since then, they have become a fundamental tool in signal analysis; see, for instance, the textbooks~\cite{FeichtingerNowak,Wong1,Wong2} and the survey~\cite{Englis}.

The classical definition of localization operators~\cite{corderogrochenig} relies on the \textit{short-time Fourier transform} (STFT). 
Given a signal \( f \in \mathcal{S}'(\mathbb{R}^d) \), the STFT with respect to a nonzero window function \( g \in \mathcal{S}(\mathbb{R}^d) \) is defined by
\begin{equation}\label{eq:stft}
	V_g f(x,\xi) = \int_{\mathbb{R}^d} f(t)\,\overline{g(t-x)}\,e^{-2\pi i t\xi}\,dt
	= \langle f, M_\xi T_x g\rangle, \quad (x,\xi)\in\mathbb{R}^{2d},
\end{equation}
where \(M_{\xi}\) and \(T_x\) denote the modulation and translation operators:
\[
M_{\xi}f(t)= e^{2\pi i \xi t} f(t),\qquad T_x f(t)=f(t-x),
\]
and their composition \(\pi(x,\xi)=M_\xi T_x\) is called the \textit{time--frequency shift}.

For two windows \( \varphi_1, \varphi_2 \in \mathcal{S}(\mathbb{R}^d)\setminus\{0\} \) and a symbol \( a \in \mathcal{S}'(\mathbb{R}^{2d}) \), 
the \textit{localization operator} is defined as
\begin{equation}\label{eq:loc}
	A_{a}^{\varphi_1,\varphi_2} f(t)
	= \int_{\mathbb{R}^{2d}} a(x,\omega)\,V_{\varphi_1} f(x,\omega)\,M_\omega T_x \varphi_2(t)\,dx\,d\omega, 
	\qquad t\in\mathbb{R}^d,
\end{equation}
and it acts continuously from \(\mathcal{S}(\mathbb{R}^d)\) to \(\mathcal{S}'(\mathbb{R}^d)\).

The \textit{Weyl quantization} of a symbol \( \sigma \in \mathcal{S}'(\mathbb{R}^{2d}) \) is defined by
\begin{equation}\label{eq:Weyl}
	\mathrm{Op}_{w}(\sigma)f(x)=
	\int_{\mathbb{R}^{2d}} e^{2\pi i (x-y)\xi}\,
	\sigma\!\left(\frac{x+y}{2},\xi\right) f(y)\,dy\,d\xi.
\end{equation}
Its associated time--frequency representation is the (cross-)Wigner distribution,
\begin{equation}\label{CWD}
	W(f,g)(x,\xi)=
	\int_{\mathbb{R}^d} f\!\left(x+\frac{t}{2}\right)
	\,\overline{g\!\left(x-\frac{t}{2}\right)}\,e^{-2\pi i t\xi}\,dt,
\end{equation}
so that
\begin{equation}\label{I7}
	\langle \operatorname{Op}_w(\sigma)f,g\rangle=\langle \sigma, W(g,f)\rangle.
\end{equation}
Every localization operator admits a Weyl representation:
\begin{equation}\label{eq:loc-Weyl}
	A^{\varphi_1,\varphi_2}_a
	= \operatorname{Op}_w(a * W(\varphi_2, \varphi_1)),
\end{equation}
see~\cite{folland89} for the Gaussian case and~\cite{BCG} for the general one.

More recently, Bastianoni and Teofanov~\cite{BastianoniTeofanov} extended this formula to the broader class of $\tau$-operators.  
For $\tau\in\mathbb{R}$, the (cross-)$\tau$-Wigner distribution is defined by
\begin{equation}\label{tau-Wigner distribution}
	W_{\tau}(f,g)(x,\xi)
	= \int_{\mathbb{R}^d} e^{-2\pi i t\xi}\,f(x+\tau t)\,
	\overline{g(x-(1-\tau)t)}\,dt,
\end{equation}
and the associated $\tau$-quantization satisfies
\begin{equation}\label{C21}
	\langle \operatorname{Op}_\tau(a)f,g\rangle
	= \langle a, W_\tau(g,f)\rangle.
\end{equation}
They proved that
\begin{equation}\label{eq:tau}
	A_{a}^{\varphi_1,\varphi_2}
	=  \operatorname{Op}_\tau(a * W_\tau(\varphi_2, \varphi_1)),\qquad \tau\in[0,1],
\end{equation}
and the same argument extends to all $\tau\in\mathbb{R}$.  
A natural question arises: \textit{does this equality extend to the full class of metaplectic pseudodifferential operators} introduced in~\cite{CR2021}?

To address this question, we recall the definition of the $\mathcal{A}$-Wigner distributions, which include the $\tau$-Wigner family as a special case.  
Let $\widehat{\cA}\in\mathrm{Mp}(2d,\mathbb{R})$ be a metaplectic operator with projection $\mathcal{A}\in\mathrm{Sp}(2d,\mathbb{R})$.  
The \textit{$\mathcal{A}$-Wigner distribution} (or \textit{metaplectic Wigner distribution}) is defined by
\[
W_\mathcal{A}(f,g):=\widehat{\cA}(f\otimes\overline{g}),\qquad f,g\in L^2(\mathbb{R}^d),
\]
which generalizes several classical time--frequency representations such as the STFT and the $\tau$-Wigner distributions, see the next section for details.

A key feature of $\tau$-Wigner distributions is their \textit{covariance property}:
\[
W_\tau(\pi(z)f,\pi(z)g)=T_zW_\tau(f,g), \qquad z\in\mathbb{R}^{2d}.
\]
This property extends to a subclass of metaplectic Wigner distributions, called \textit{covariant metaplectic Wigner distributions}, satisfying
\begin{equation}\label{eq:covarianceid}
	W_{\mathcal{A}}(\pi(z)f,\pi(z)g)=T_zW_{\mathcal{A}}(f,g).
\end{equation}

Given such a covariant $W_\mathcal{A}$, the corresponding \textit{metaplectic pseudodifferential operator} is
\begin{equation}\label{capseudo}
	\langle \operatorname{Op}_\mathcal{A}(\sigma)f,g\rangle
	= \langle \sigma, W_\mathcal{A}(g,f)\rangle,
\end{equation}
for $\sigma\in\mathcal{S}'(\mathbb{R}^{2d})$.

We can now generalize the result of~\cite{BastianoniTeofanov} as follows.

\begin{theorem}\label{main}
	Let \( \varphi_1,\varphi_2\in\mathcal{S}(\mathbb{R}^d)\setminus\{0\}\), \(a\in\mathcal{S}'(\mathbb{R}^{2d})\), and \(\mathcal{A}\in \mathrm{Sp}(2d,\mathbb{R})\).  
	If \(W_\mathcal{A}\) is covariant, then
	\begin{equation}\label{eq:a-loc}
		A_{a}^{\varphi_1,\varphi_2}
		= \operatorname{Op}_\mathcal{A}(a*W_\mathcal{A}(\varphi_2,\varphi_1)).
	\end{equation}
	Conversely, if this identity holds for all \(a\in\cS'(\rdd),\varphi_1,\varphi_2\in\cS(\rd)\), then \(W_\mathcal{A}\) is covariant.
\end{theorem}

This correspondence motivates the study of the more general class of \textit{$\mathcal{A}$-localization operators}, defined by
\begin{equation}\label{AlocopElena}
	A_{a,\mathcal{A}}^{\varphi_1,\varphi_2}
	:= \operatorname{Op}_\mathcal{A}(a * W_\mathcal{A}(\varphi_2,\varphi_1)).
\end{equation}
These operators act continuously from $\mathcal{S}(\mathbb{R}^d)$ to $\mathcal{S}'(\mathbb{R}^d)$. 
In this paper, we analyze the mapping $(a,\varphi_1,\varphi_2)\mapsto A_{a,\mathcal{A}}^{\varphi_1,\varphi_2}$,
derive boundedness results on modulation spaces, and provide Schatten-class criteria,
extending the classical framework of Cordero and Gröchenig~\cite{CorderoGrochenig2005}.

\vspace{0.1truecm}
\noindent\textit{Outline of the paper.}
Section~\ref{sec:preliminaries} recalls the main tools from time-frequency analysis, including modulation spaces, the symplectic and metaplectic groups, and the class of metaplectic Wigner distributions. 
In Section~\ref{sec:Aloc} we prove our main structural result, Theorem~\ref{main}, which characterizes the covariance property of metaplectic Wigner distributions in terms of the representation formula~\eqref{eq:a-loc}. 
We then provide explicit expressions for the Schwartz kernel of $\mathcal{A}$-localization operators, including the case of totally Wigner-decomposable symplectic matrices. 
Finally, we deal with the continuity and Schatten-class results on modulation spaces, obtained via the Weyl correspondence and convolution estimates for modulation spaces.

\section{Preliminaries}\label{sec:preliminaries}
\textbf{Notation.} We denote by 
$xy=x\cdot y$ the scalar product on $\Ren$.  The space   $\sch(\Ren)$ is the Schwartz class whereas its dual $\sch'(\Ren)$ is the space of temperate distributions. The brackets  $\la f,g\ra$ are the extension to $\sch' (\Ren)\times\sch (\Ren)$ of the inner product $\la f,g\ra=\int f(t){\overline {g(t)}}dt$ on $L^2(\Ren)$ (conjugate-linear in the second component). A point in the phase space (or \tf\ space) is written as
$z=(x,\xi)\in\rdd$, and  the corresponding phase-space shift (\tfs )
acts as
$
\pi (z)f(t) = e^{2\pi i \xi t} f(t-x),\ t\in\rd
$, that is the composition of the translation and modulation operators \begin{equation*}
    T_xf(t):=f(t-x),\quad M_{\xi}f(t):=e^{2\pi i t\xi} f(t),\quad t,x,\xi\in\rd.
\end{equation*}

The notation $f\lesssim g$ means that there exists $C>0$ such that $ f(x)\leq Cg(x)$  for every $x$. The symbol $\lesssim_t$ is used to stress that $C=C(t)$. If $ g\lesssim f\lesssim g$ (equivalently, $ f \lesssim g\lesssim f$), we write $f\asymp g$. Given two measurable functions $f,g:\rd\to\bC$, we set $f\otimes g(x,y):=f(x)g(y)$. If $X(\rd)$ is any among $L^2(\rd),\cS(\rd),\cS'(\rd)$, $X\otimes X$ is the unique completion of $\text{span}\{x\otimes y : x\in X(\rd)\}$ with respect to the (usual) topology of $X(\rdd)$. Thus, the operator $f\otimes g\in\cS'(\rdd)$  characterized by its action on $\varphi\otimes\psi\in\cS(\rdd)$
\[
	\la f\otimes g,\varphi\otimes\psi\ra = \la f,\varphi\ra\la g,\psi\ra,\quad  \forall f,g\in\cS'(\rd),
\]
extends uniquely to a tempered distribution of $\cS'(\rdd)$. The subspace $\text{span}\{f\otimes g: f,g\in\cS'(\rd)\}$ is dense in $\cS'(\rdd)$.

$GL(d,\bR)$ stands for the group of $d\times d$ invertible matrices, whereas $Sym(d,\bR)=\{C\in\bR^{d\times d} \ : \ C \ is \ symmetric\}$.

\subsection{Schatten-von Neumann Classes}
Let $\cH$ be a separable Hilbert space and $T:\cH\to \cH$ a compact operator. Then, $T^*T:\cH\to \cH$ is a compact, self-adjoint, non-negative operator. Hence, we can define its absolute value $|T|:=(T^*T)^{1/2}$ which is still compact, self-adjoint and non-negative on $\cH$. Therefore, by the spectral theorem we can find an orthonormal basis $(e_n)_n$ for $\cH$ consisting of eigenvectors of $|T|$. The corresponding eigenvalues $s_1(T)\geq s_2(T)\geq\dots\geq s_n(T)\geq \dots \geq 0$, are called the singular values of $T$. If $0<p<\i$ and the sequence of singular values is in $\ell^p$, then $T$ is said to belong to the Schatten-von Neumann class $\cS_p(\cH)$. If $1\leq p<\i$, a norm is associated to $\cS_p(\cH)$ by

\begin{equation}\label{schattennorm}
\norm{T}_{\cS_p}:=\bigg(\displaystyle\sum_{n=1}^\i s_n(T)^p\bigg)^{\frac{1}{p}}.
\end{equation}

If $1\leq p<\i$, then $(\cS_p(\cH),\norm{\cdot}_{\cS_p})$ is a Banach space whereas, for $0<p<1$, it is a quasi-Banach space since the quantity $\norm{T}_{\cS_p}$ defined in \eqref{schattennorm} is only a quasinorm. In this work we will only work with the Schatten classes $\cS_p(L^2(\rdd))$ which will be simply denoted by $\cS_p$.


If  $0<p,q\leq\infty$ and $f:\rdd\to\bC$ measurable, we set 
\[
	\norm{f}_{L^{p,q}}:=\left(\int_{\rd}\left(\int_{\rd}|f(x,y)|^p dx\right)^{\frac{q}{p}}dy\right)^{\frac{1}{q}},
\]
with the obvious adjustments when $\max\{p,q\}=\infty$. The space of measurable functions $f$ having $\norm{f}_{L^{p,q}}<\infty$ is denoted by $L^{p,q}(\rdd)$. 
%

\subsection{Time-frequency analysis tools}\label{subsec:23}
In this work, the Fourier transform of $f\in \cS(\rd)$ is normalized as
\[
\cF f=\hat f(\xi)=\int_{\rd} f(x)e^{-2\pi i\xi x}dx, \qquad \xi\in\rd.
\]
If $f\in\cS'(\rd)$, the Fourier transform of $f$ is defined by duality as the tempered distribution characterized by
\[
\langle \hat f,\hat\varphi\rangle=\langle f,\varphi\rangle, \qquad \varphi\in\cS(\rd).
\]
The operator $\cF$  is a surjective automorphism of $\cS(\rd)$ and $\cS'(\rd)$, as well as a surjective isometry of $L^2(\rd)$.
If $f\in\cS'(\rdd)$, we set $\cF_2f$, the partial Fourier transform with respect to the second variables: $$\cF_2(f\otimes g)=f\otimes\hat g,\quad f,g\in\cS'(\rd).$$
The \textit{short-time Fourier transform} of $f\in L^2(\rd)$ with respect to the window $g\in L^2(\rd)$ is defined in \eqref{eq:stft}.

In information processing $\tau$-Wigner distributions  ($\tau\in\bR$) play a crucial role \cite{ZJQ21}. They are defined in \eqref{tau-Wigner distribution}.
For $\tau=1/2$ we have the \textit{Wigner distribution}, defined in \eqref{CWD}.

\subsection{Modulation  spaces \cite{KB2020,F1,Feichtinger_1981_Banach,book,Galperin2004,Kobayashi2006,PILIPOVIC2004194}} \label{subsec:MSs}
For $0<p,q\leq\infty$,   $g\in\cS(\rd)\setminus\{0\}$, the \textit{modulation space} $M^{p,q}(\rd)$ is defined as the space of tempered distributions $f\in\cS'(\rd)$ such that $$\norm{f}_{M^{p,q}}:=\Vert V_gf\Vert_{L^{p,q}}<\infty.$$ If $\min\{p,q\}\geq1$, the quantity $\norm{\cdot}_{M^{p,q}}$ is a norm, otherwise a quasi-norm. Different windows give  equivalent (quasi-)norms. Modulation spaces are (quasi-)Banach spaces, enjoying the inclusion properties:
if $0<p_1\leq p_2\leq\infty$ and $0<q_1\leq q_2\leq\infty$  $$ \cS(\rd)\hookrightarrow M^{p_1,q_1}(\rd)\hookrightarrow M^{p_2,q_2}(\rd)\hookrightarrow\cS'(\rd).$$ In particular, $M^1(\rd)\hookrightarrow M^{p,q}(\rd)$ and $\min\{p,q\}\geq1$. 
If $1\leq p,q<\infty$, $(M^{p,q}(\rd))'=M^{p',q'}(\rd)$, where $p'$ and $q'$ denote the Lebesgue dual exponents of $p$ and $q$, respectively.

\subsection{The symplectic group $Sp(d,\mathbb{R})$ and the metaplectic operators}\label{subsec:26}
	A matrix $A\in\bR^{2d\times 2d}$ is symplectic, write $A\in Sp(d,\bR)$, if 
	\begin{equation}\label{fundIdSymp}
	A^TJA=J,\end{equation} where $J$ is the standard symplectic matrix:
    \begin{equation}\label{defJ}
        J=\begin{pmatrix}
	        0_{d\times d}&I_{d\times d}\\-I_{d\times d}&0_{d\times d}
	    \end{pmatrix}.
	\end{equation}
	
  \begin{remark}
      It is easy to check that $Sp(d,\bR)$ is a subgroup of $SL(2d,\bR)$ (see e.g. \cite{degosson}), in particular, if we write $\cA\in Sp(d,\bR)$ with block decomposition:\begin{equation*}
          \cA=\begin{pmatrix}
              A&B\\C&D,
          \end{pmatrix},\quad A,B,C,D\in \bR^{d\times d},
      \end{equation*}
      then the inverse of $\cA$ is given by:\begin{equation}\label{inverseA}
          \cA^{-1}=\begin{pmatrix}
              D^T& -B^T\\
              -C^T & A^T
          \end{pmatrix}.
      \end{equation}
  \end{remark}
  
	For $E\in GL(d,\bR)$ and $C\in Sym(2d,\bR)$, define:
	\begin{equation}\label{defDLVC}
		\cD_E:=\begin{pmatrix}
			E^{-1} & 0_{d\times d}\\
			0_{d\times d} & E^T
		\end{pmatrix} \qquad \text{and} \qquad V_C:=\begin{pmatrix}
			I_{d\times d} & 0\\ C & I_{d\times d}
		\end{pmatrix}.
	\end{equation}
	The matrices $J$, $V_C$ ($C$ symmetric), and $\cD_E$ ($E$ invertible) generate the group $Sp(d,\bR)$.\\

Recall  the Schr\"odinger representation $\rho$  of the Heisenberg group: $$\rho(x,\xi;\tau)=e^{2\pi i\tau}e^{-\pi i\xi x}\pi(x,\xi),$$ for all $x,\xi\in\rd$, $\tau\in\bR$. We will use the  property:  for all $f,g\in L^2(\rd)$, $z=(z_1,z_2),w=(w_1,w_2)\in\rdd$,
\[
	\rho(z;\tau)f\otimes\rho(w;\tau)g=e^{2\pi i\tau}\rho(z_1,w_1,z_2,w_2;\tau)(f\otimes g).
\]
For every $A\in Sp(d,\bR)$, $\rho_A(x,\xi;\tau):=\rho(A (x,\xi);\tau)$ defines another representation of the Heisenberg group that is equivalent to $\rho$, i.e., there exists a unitary operator $\hat A:L^2(\rd)\to L^2(\rd)$ such that:
\begin{equation}\label{muAdef}
	\hat A\rho(x,\xi;\tau)\hat A^{-1}=\rho(A(x,\xi);\tau), \qquad  x,\xi\in\rd, \ \tau\in\bR.
\end{equation}
This operator is not unique: if $\hat A'$ is another unitary operator satisfying (\ref{muAdef}), then $\hat A'=c\hat A$, for some constant $c\in\bC$, $|c|=1$. The set $\{\hat A : A\in Sp(d,\bR)\}$ is a group under composition and it admits the metaplectic group, denoted by $Mp(d,\bR)$, as subgroup. It is a realization of the two-fold cover of $Sp(d,\bR)$ and the projection:
 \begin{equation}\label{piMp}
	\pi^{Mp}:Mp(d,\bR)\to Sp(d,\bR)
\end{equation} is a group homomorphism with kernel $\ker(\pi^{Mp})=\{-id_{{L^2}},id_{{L^2}}\}$.

Throughout this paper, if $\hat A\in Mp(d,\bR)$, the matrix $A$ will always be the unique symplectic matrix such that $\pi^{Mp}(\hat A)=A$.

%
In what follows we list some important examples of metaplectic operators we are going to use next.
\begin{example}\label{es22} Consider the symplectic matrices $J$, $\cD_L$ and $V_C$  defined  in (\ref{defJ}) and (\ref{defDLVC}), respectively. Then,
	\begin{enumerate}
		\item[\it (i)] $\pi^{Mp}(\cF)=J$;
		\item[\it (ii)] if $\mathfrak{T}_E:=|\det(E)|^{1/2}\,f(E\cdot)$, then $\pi^{Mp}(\mathfrak{T}_E)=\cD_E$;
		\item[\it (iii)] if $\cF_2 $ is the Fourier transform with respect to the second variables, then $\pi^{Mp}(\cF_2)=\cA_{FT2}$, where $\cA_{FT2}\in Sp(2d,\bR)$ is the $4d\times4d$ matrix with block decomposition
		\begin{equation}\label{AFT2}
		\cA_{FT2}:=\begin{pmatrix}
			I_{d\times d} & 0_{d\times d} & 0_{d\times d} & 0_{d\times d}\\
			0_{d\times d} & 0_{d\times d} & 0_{d\times d} & I_{d\times d} \\
			0_{d\times d} & 0_{d\times d} & I_{d\times d} & 0_{d\times d}\\
			0_{d\times d} & -I_{d\times d} & 0_{d\times d} & 0_{d\times d}
		\end{pmatrix}.
		\end{equation}
	\end{enumerate}

\end{example}

\subsection{Metaplectic Wigner distributions}
Let $\hat\cA\in Mp(2d,\bR)$. The \textbf{metaplectic Wigner distribution} associated to $\hat\cA$ is defined as:\begin{equation}
    W_{\cA}(f,g):=\hat \cA(f\otimes \overline{g}),\quad f,g\in L^2(\rd).
\end{equation} \label{WApre}
The most popular  time-frequency representations  fall in the class of metaplectic Wigner distributions. Namely, the STFT can be represented as \begin{equation}\label{eq:stft-metap}
	 V_gf=\hat A_{ST}(f\otimes\bar g)
\end{equation}  where
\begin{equation}\label{AST}
	A_{ST}=\begin{pmatrix}
		I_{d\times d} & -I_{d\times d} & 0_{d\times d} & 0_{d\times d}\\
		0_{d\times d} & 0_{d\times d} & I_{d\times d} & I_{d\times d}\\
		0_{d\times d} & 0_{d\times d} & 0_{d\times d} & -I_{d\times d}\\
		-I_{d\times d} & 0_{d\times d} & 0_{d\times d} &0_{d\times d}
	\end{pmatrix}.
\end{equation}
The $\tau$-Wigner distribution defined in \eqref{tau-Wigner distribution} can be recast as $W_\tau(f,g)=\hat A_\tau(f\otimes\bar g)$, with
\begin{equation}\label{Atau}
	A_\tau=\begin{pmatrix}
		(1-\tau)I_{d\times d} & \tau I_{d\times d} & 0_{d\times d} & 0_{d\times d}\\
		0_{d\times d} & 0_{d\times d} & \tau I_{d\times d} & -(1-\tau)I_{d\times d}\\
		0_{d\times d} & 0_{d\times d} & I_{d\times d} & I_{d\times d}\\
		-I_{d\times d} & I_{d\times d} & 0_{d\times d} & 0_{d\times d}
	\end{pmatrix}.
\end{equation}
Similarly to the STFT, these time-frequency representations enjoy a \emph{reproducing formula}, cf. \cite[Lemma 3.6]{CGshiftinvertible}:
\begin{lemma}\label{intertFormula}
Consider $\hat\cA\in Mp(2d,\bR)$, with $\pi^{Mp}(\hat{\cA})=\cA\in Sp(2d,\bR)$, $\gamma,g\in\cS(\rd)$  such that $\langle \gamma,g\rangle\neq0$ and $f\in\cS'(\rd)$. Then,
	\begin{equation}\label{e6}
		W_\cA(f,g)=\frac{1}{\langle\gamma, g\rangle}\int_{\rdd}V_gf(w)W_\cA(\pi(w)\gamma,g) dw,
	\end{equation}
	with equality in $\cS'(\rdd)$, the integral being intended in the weak sense.
\end{lemma}
From the right-hand side we infer that the key point becomes the action of $W_\cA$ on the time-frequency shift $\pi(w)$, which can be computed explicitly. 
For $\cA\in Sp(2d,\bR)$, it will be useful to consider its block decomposition:
 \begin{equation}\label{blockA}
	\cA=\begin{pmatrix}
		A_{11} & A_{12} & A_{13} & A_{14}\\
		A_{21} & A_{22} & A_{23} & A_{24}\\
		A_{31} & A_{32} & A_{33} & A_{34}\\
		A_{41} & A_{42} & A_{43} & A_{44}
	\end{pmatrix}.
\end{equation}

We recall the following continuity properties.
\begin{proposition}\label{prop25}
	Let $W_\cA$ be a metaplectic Wigner distribution. Then,\\
	 $W_\cA:L^2(\rd)\times L^2(\rd) \to L^2(\rdd) $ is bounded.
	 The same result holds if we replace  $L^2$ by $\cS$ or $\cS'$.
\end{proposition}
Since metaplectic operators are unitary, for all $f_1,f_2,g_1,g_2\in L^2(\rd)$,
\begin{equation}\label{Moyal}
	\la W_\cA(f_1,f_2),W_\cA(g_1,g_2)\ra = \la f_1,g_1\ra \overline{\la f_2,g_2\ra}.
\end{equation}

$W_{\A}$ is said to be \textit{covariant} if it satisfies the covariance property in \eqref{eq:covarianceid}.
 The following proposition provides a complete characterization of symplectic matrices that give rise to covariant metaplectic Wigner distribution.
\begin{proposition}[Proposition 4.4 in \cite{CR2021}]\label{propcovariant}
    Let $\A\in Sp(2d,\bR)$, then $W_{\cA}$ is covariant if and only if the block decomposition \eqref{blockA} of $\A$ is of the form:\begin{equation}\label{eq:covariant}
        \A=\begin{pmatrix}
            A_{11}&I_{d\times d}-A_{11}&A_{13}&A_{13}\\A_{21}&-A_{21}&I_{d\times d}-A_{11}^T&-A_{11}^T\\0_{d\times d}&0_{d\times d}&I_{d\times d}&I_{d\times d}\\
            -I_{d\times d}&I_{d\times d}&0_{d\times d}&0_{d\times d}
        \end{pmatrix},
    \end{equation}
    with $A_{13}=A_{13}^T$ and $A_{21}=A_{21}^T$.
\end{proposition}

\subsection{\textbf{Metaplectic pseudodifferential operators}}
These pseudodifferential operators were introduced in \cite{CR2021} and generalize the classical ones.
	{\begin{definition}\label{defMetaplPsiDo}
	Let $a\in\mathcal{S}'(\mathbb{R}^{2d})$. The \textbf{metaplectic pseudodifferential operator} with \textbf{symbol} $a$ and symplectic matrix $\cA$ is the operator $\operatorname{Op}_\mathcal{A}(a):\mathcal{S}(\mathbb{R}^d)\to\mathcal{S}'(\mathbb{R}^{d})$ such that
\begin{equation}\label{eq:metapseudo}
			\langle \operatorname{Op}_\mathcal{A}(a)f,g\rangle=\langle a,W_\mathcal{A}({g},f)\rangle, \quad g\in\mathcal{S}(\mathbb{R}^d).	
\end{equation}
	\end{definition}
	Observe that this operator is well defined by Proposition \ref{prop25}. Moreover, when the context requires to stress the matrix $\cA$ that defines $\operatorname{Op}_\cA$, we refer to $\operatorname{Op}_\cA$ to as the \textbf{$\cA$-pseudodifferential operator} with {symbol} $a$. }
	{\begin{remark}
		In principle, the full generality of metaplectic framework provides a wide variety of unexplored time-frequency representations that fit many different contexts. Namely, in Definition \ref{defMetaplPsiDo}, the symplectic matrix $\cA$ plays the role of a quantization and the quantization of a pseudodifferential operator is typically chosen depending on the the properties that must be satisfied in a given setting.
	\end{remark}}
	
	\begin{example}
		Definition \ref{defMetaplPsiDo} in the case of $\cA_{1/2}\in Sp(2d,\mathbb{R})$ provides the well-known Weyl quantization for pseudodifferential operators, cf. \eqref{I7} in the introduction.
	\end{example}
	
 The following issue shows how the symbols of {metaplectic pseudodifferential operators} change when we modify the symplectic matrix. 
	\begin{lemma}[Lemma 3.2. in \cite{CGR2022}]\label{lemmaComm}
		Consider $\mathcal{A},\mathcal{B}\in Sp(2d,\mathbb{R})$ and $a,b\in\mathcal{S}'(\mathbb{R}^{2d})$. Then, 
		\begin{equation}\label{changeMatrix}\operatorname{Op}_{\mathcal{A}}(a)=\operatorname{Op}_\mathcal{B}(b)\quad
		\iff \quad b= \hat{\mathcal{B}}\hat{\mathcal{A}}^{-1}(a).
		\end{equation}
	\end{lemma}

As a direct consequence of Lemma \ref{lemmaComm} we get the following corollary, which provides the distributional kernel of $\operatorname{Op}_\cA$.
	
	\begin{cor} \label{cor:kernel}Consider $\cA\in Sp(2d,\mathbb{R})$, $a\in \mathcal{S}'(\rdd)$. Then, for all $f,g\in\mathcal{S}(\rd)$,
	\begin{equation}\label{kernelMPO}
		\langle \operatorname{Op}_\cA(a)f,g\rangle=\langle k_\cA(a),g\otimes \bar f\rangle,		
	\end{equation}
	where the kernel is given by $k_\cA(a)=\hat \cA^{-1}a$. 
	\end{cor}
	\begin{proof}
	Plug $\mathcal{B}=I_{4d\times 4d}$ into (\ref{changeMatrix}) to get (\ref{kernelMPO}). 
	\end{proof}
	
Another immediate consequence of Lemma \ref{lemmaComm}
is that every metaplectic pseudodifferential operator of the form $\operatorname{Op}_{\cA}(a)$ can be written as a Weyl operator $\operatorname{Op}_w(\sigma)$ with symbol $\sigma=\hat \cA_{1/2}\hat \cA^{-1}(a)$, which is called the \textit{Weyl symbol} of $\operatorname{Op}_{\cA}(a)$. We recall an important theorem concerning sufficient conditions for $\operatorname{Op}_w(\sigma)$ to belong to the Schatten classes, for details see Theorem 3.1. in \cite{corderogrochenig}.
\begin{theorem}\label{teoweylsymb}
    \begin{itemize}
        \item[i)] If $1\leq p\leq 2$ and $\sigma\in M^p(\rdd)$, then $\operatorname{Op}_w(\sigma)\in \cS_p$ and $\norm{\operatorname{Op}_w(\sigma)}_{\cS_p}\lesssim \norm{\sigma}_{M^{p}}$. 
        \item[ii)] If $2\leq p\leq \i$ and $\sigma \in M^{p,p'}(\rdd)$, then $\operatorname{Op}_w(\sigma)\in \cS_p$ and $\norm{\operatorname{Op}_w(\sigma)}_{\cS_p}\lesssim \norm{\sigma}_{M^{p,p'}}$. 
    \end{itemize}

\end{theorem}


\section{ $\cA$-localization operators}\label{sec:Aloc}

	Let \( \varphi_1, \varphi_2 \in \mathcal{S}(\mathbb{R}^d) \setminus \{0\} \) and \( a \in \mathcal{S}'(\mathbb{R}^{2d}) \). The \textit{localization operator} \( A_{a}^{\varphi_1,\varphi_2} \) is defined in \eqref{eq:loc}. 
	For \( f, g \in \mathcal{S}(\mathbb{R}^d) \), the operator can also be written in the weak form
	\[
	\langle A_{a}^{\varphi_1,\varphi_2} f, g \rangle = \langle a, \overline{V_{\varphi_1} f }\cdot V_{\varphi_2} g \rangle,
	\]
	where the duality extends the inner product on $L^2$.
	
	We recall Proposition 2.16 in \cite{BastianoniTeofanov}:
	\begin{proposition}
		Let \( \varphi_1, \varphi_2 \in \mathcal{S}(\mathbb{R}^d) \setminus \{0\} \), \( a \in \mathcal{S}'(\mathbb{R}^{2d}) \), and \( \tau \in [0,1].\) Then the localization operator  $ A_{a}^{\varphi_1,\varphi_2}$ coincides with the \(\tau\)-localization operator:
	\[
	A_{a}^{\varphi_1,\varphi_2} = A_{a,\tau}^{\varphi_1,\varphi_2},
	\]
	where
	\[
	A_{a,\tau}^{\varphi_1,\varphi_2} := \operatorname{Op}_{\tau}(a * W_\tau(\varphi_2, \varphi_1)).
	\]
	
	\end{proposition}
    \begin{remark}
      We observe that $$  \operatorname{Op}_{\tau}(a * W_\tau(\varphi_2, \varphi_1)):={Op}_{\cA_\tau}(a \ast W_{\cA_\tau}(\varphi_2, \varphi_1)),$$
      where  the symplectic matrix $\cA_\tau$ is defined in \eqref{Atau}. Furthermore,  the result above holds for every $\tau\in\bR$.
    \end{remark}
	$\tau$-Wigner distributions are particular cases of covariant metaplectic Wigner distributions for $\tau\in\bR$. We state the following lemma which allows us to generalize the previous result.
\begin{lemma}\label{lemmacovariantconvolution}
    Let $W_{\A}$ be a covariant metaplectic Wigner distribution with projection $\A\in Sp(2d,\bR)$, then:\begin{equation}\label{covariantconvol}
        W_{\A}(f_1,g_1)\ast W_{\A}(f_2,g_2)^*=W(f_1,g_1)\ast W(f_2,g_2)^*,
    \end{equation}
   for every $f_i,g_i\in L^2(\rd),\ i=1,2$, where we set $f^*(t):=\overline{f}(-t)$.
\end{lemma}
\begin{proof}
    For every $f_i,g_i\in L^2(\rd),\ i=1,2$, we observe that $W(f_i,g_i)\in L^2(\rdd)$ and compute the convolution of the cross Wigner distributions:\begin{align*}
        \big(W(f_1,g_1)\ast W(f_2,g_2)^*\big)(w)=&\int_{\rdd}W(f_1,g_1)(u)W(f_2,g_2)^*(w-u)du\\=&\int_{\rdd}W(f_1,g_1)(u)\overline{W(f_2,g_2)}(w-u)du\\
        =&\int_{\rdd}\!\!W(f_1,g_1)(u)\!\overline{W(\pi(w)f_2,\pi(w)g_2)}(u)du.
    \end{align*}
   The last equality follows from the covariance property for the Wigner distribution, see \eqref{eq:covarianceid} for $\cA=\cA_{1/2}$. Using  Moyal's identity for $W$ and $W_{\A}$,
    \begin{align*}
        \int_{\rdd}W(f_1,g_1)(u)\overline{W(\pi(w)f_2,\pi(w)g_2)}(u)du&=\la W(f_1,g_1),W(\pi(w)f_2,\pi(w) g_2)\ra_{L^2(\rdd)}\\
        &=\la f_1,\pi(w) f_2\ra_{L^2(\rd)}\overline {\la g_1,\pi(w) g_2\ra}_{L^2(\rd)}\\
        &=\la W_{\A}(f_1,g_1),W_{\A}(\pi(w)f_2,\pi(w) g_2)\ra_{L^2(\rdd)}.
    \end{align*}
  Since $W_{\A}$ is covariant we can write\begin{align*}
       \la W_{\A}(f_1,g_1),W_{\A}(\pi(w)f_2,\pi(w) g_2)\ra_{L^2(\rdd)}  =&\int_{\rdd}W_{\A}(f_1,g_1)(u)\overline{W_{\A}(\pi(w)f_2,\pi(w)g_2)}(u)du\\
      =&\int_{\rdd}W_{\cA}(f_1,g_1)(u)\overline{W_{\cA}(f_2,g_2)}(u-w)du\\
       =&\int_{\rdd}W_{\cA}(f_1,g_1)(u)W_{\cA}(f_2,g_2)^*(u-w)du\\
       =& \big(W_{\A}(f_1,g_1)\ast W_{\A}(f_2,g_2)^*\big)(w).
    \end{align*}
    This concludes the proof.
\end{proof}

The sufficient conditions in Theorem \ref{main} can be obtained as an easy consequence of Lemma \ref{lemmacovariantconvolution}.
 \begin{proof}[\bf Proof of the sufficient condition of Theorem \ref{main}]
 	  Assume $W_{\A}$ is covariant. For every $f,g\in \cS(\rd)$, we use the connection between localization and Weyl  operators \eqref{eq:loc-Weyl} and then its weak definition in \eqref{I7} to write
 	  \begin{align*}
      \la A_a^{\varphi_1,\varphi_2}f,g\ra=&\la \operatorname{Op}_w(a\ast W(\varphi_2,\varphi_1))f,g\ra =\la a\ast W(\varphi_2,\varphi_1), W(g,f)\ra\\
      =& \la a, W(g,f)\ast W(\varphi_2,\varphi_1)^*\ra.
  \end{align*}
 Since $W_{\A}$ is covariant we can apply Lemma \ref{lemmacovariantconvolution}:
  \begin{align*}
      \la a, W(g,f)\ast W(\varphi_2,\varphi_1)^*\ra
      =& \la a, W_{\A}(g,f)\ast W_{\A}(\varphi_2,\varphi_1)^*\ra\\
      =& \la a\ast W_{\A}(\varphi_2,\varphi_1),W_{\A}(g,f)\ra\\
      =& \la \operatorname{Op}_{\A}(a\ast W_{\A}(\varphi_2,\varphi_1))f,g \ra\\
      =& \la A_{a,\A}^{\varphi_1,\varphi_2}f,g\ra,
  \end{align*}
  where in the last-but-one row we applied the definition of metaplectic pseudodifferential operator in \eqref{eq:metapseudo} and in the last line of $\cA$-localization operator in \eqref{AlocopElena}.\end{proof}

 \begin{proof}[\bf Proof of the vice versa of Theorem \ref{main}]
 By exploiting Lemma \ref{lemmaComm} we rewrite the condition \eqref{eq:a-loc} as:\begin{equation}\label{eqalocnew}
     a\ast W_{\cA}(\varphi_2,\varphi_1)=\hat \cA\hat \cA_{1/2}^{-1}(a\ast W(\varphi_2,\varphi_1)),\quad \forall \varphi_1,\varphi_2\in\cS(\rd),\ \forall a\in \cS'(\rdd).
 \end{equation}
Our goal is to prove that:\begin{equation}\label{covariantAWig}
    W_{\cA}(\pi(z)f,\pi(z)g)=T_zW_{\cA}(f,g),\quad \forall f,g\in L^2(\rd),\ z\in\rdd.
\end{equation}
We recall that:\begin{equation*}
    T_x\varphi=\delta_x\ast \varphi,\quad \forall \varphi\in\cS(\rd),\forall x\in\rd.
\end{equation*}
Let $f,g\in\cS(\rd)$ and $z\in\rdd$,   using the covariance property for the Wigner distribution: $T_zW(f,g)=W(\pi(z)f,\pi(z)g)$, we obtain
\begin{align*}
    T_zW_{\cA}(f,g)=&\delta_z\ast W_{\cA}(f,g)  = \hat \cA \hat \cA_{1/2}^{-1}(\delta_z\ast W(f,g))\\
    =& \hat \cA \hat \cA_{1/2}^{-1}(T_zW(f,g))
    =\hat \cA\hat \cA^{-1}_{1/2}(W(\pi(z)f,\pi(z)g))\\
    =& \hat \cA \hat \cA_{1/2}^{-1}\hat \cA_{1/2}(\pi(z)f\otimes \overline{\pi(z)g})
    = \hat \cA(\pi(z)f\otimes \overline{\pi(z)g})\\
    =& W_{\cA}(\pi(z)f,\pi(z)g)).
\end{align*}
In conclusion, the identity \eqref{covariantAWig} is obtained by the density of $\cS(\rd)$ in $L^2(\rd)$.
 \end{proof}

From the vice versa of Theorem \ref{main} it is immediate to get the vice versa of Lemma \ref{lemmacovariantconvolution}.

\begin{lemma}\label{viceversa3.3}
    Let $\cA\in Sp(2d,\bR)$ such that \eqref{covariantconvol} holds for every $f_i,g_i\in L^2(\rd),$ $i=1,2$. Then $W_{\cA}$ is covariant.
\end{lemma}
\begin{proof}
   Fix an arbitrary symbol $a\in\cS'(\rdd)$ and arbitrary pair of windows $\varphi_1,\varphi_2\in\cS(\rd)$. By \eqref{covariantconvol},\begin{equation*}
        \la a, W_{\cA}(g,f)\ast W_{\cA}(\varphi_2,\varphi_1)^*\ra= \la a, W(g,f)\ast W(\varphi_2,\varphi_1)^*\ra,\quad \forall f,g\in \cS(\rd).
    \end{equation*}
    Hence, \begin{equation*}
        A_{a,\cA}^{\varphi_1,\varphi_2}=A_a^{\varphi_1,\varphi_2},
    \end{equation*}
    and this is true for every choice of $a,\varphi_1,\varphi_2$. Hence, by Theorem \ref{main} we conclude the thesis.
\end{proof}

 \begin{remark}
     The vice versa of Theorem \ref{main} ensures that for every non-covariant metaplectic Wigner distribution there exist $a\in\cS'(\rdd)$ and $\varphi_1,\varphi_2\in\cS(\rd)$ such that \begin{equation*}
         A_{a,\cA}^{\varphi_1,\varphi_2}\ne A_{a}^{\varphi_1,\varphi_2}.
     \end{equation*}
     However, if the symbol $a$ is the Dirac delta distribution centered at the origin, the identity  \eqref{eq:a-loc} is satisfied for every choice of $\varphi_1,\varphi_2\in\cS(\rd)$ and $\cA\in Sp(2d,\bR)$, indeed, $\forall f,g\in\cS(\rd)$, \begin{align*}
         \la  A_{\delta,\cA}^{\varphi_1,\varphi_2}f,g\ra&=\la Op_{\cA}(\delta\ast W_{\cA}(\varphi_2,\varphi_1))f,g\ra=\la Op_{\cA}( W_{\cA}(\varphi_2,\varphi_1))f,g\ra\\
         &=\la W_{\cA}(\varphi_2,\varphi_1)),W_{\cA}(g,f)\ra=\la W(\varphi_2,\varphi_1)),W(g,f)\ra\\
         &=\la Op_w(W(\varphi_2,\varphi_1))f,g\ra=\la A_{\delta}^{\varphi_1,\varphi_2}f,g\ra,
     \end{align*}
         as desired.
 \end{remark}

\subsection{Counterexamples for non-covariant matrices $\cA$}
    We exhibit here explicit calculations of convolutions of non-covariant distributions, underlying again that the identity \eqref{covariantconvol} does not hold in this case.

\begin{example}\label{counter1}
    Let $\A=I\in Sp(2d,\bR)$.  The identity $I$ is not covariant since it does not satisfy the block decomposition in \eqref{eq:covariant}.  Consider the Gaussian\begin{equation}\label{eq:phi}
    	 \phi(t):=e^{-\pi t^2},\quad t\in\rd. 
    \end{equation}
    Then, for $f=g=\varphi_1=\varphi_2=\phi$,
    \begin{equation}\label{eq:counteres1}
    	W_{I}\phi\ast W_I\phi^*\not= W\phi\ast W\phi^*.
    \end{equation}

\end{example}
  \begin{proof} we compute:\begin{align*}
       ( W_{I}\phi\ast W_I\phi^*)(t,s)=&\int_{\rdd}\phi(x)\phi(\xi)\overline{\phi}(x-t)\overline{\phi}(\xi -s)dxd\xi\\
        =& \int_{\rdd} e^{-\pi x^2}e^{-\pi \xi^2} e^{-\pi (x-t)^2}e^{-\pi (\xi-s)^2}dxd\xi\\
        =& \int_{\rd}e^{-\pi (x^2-(x-t)^2)}dx\int_{\rd} e^{-\pi (\xi^2-(\xi-s)^2)}d\xi.
    \end{align*}
    Now,\begin{align*}
        \int_{\rd}e^{-\pi (x^2-(x-t)^2)}dx= \int_{\rd} e^{-\pi (2x^2+t^2-2xt)}dx  =e^{-\pi \frac{t^2}{2}}\int_{\rd} e^{-\pi (\sqrt{2}x-\frac{t}{\sqrt{2}})^2}dx
        = 2^{-\frac{d}{2}}e^{-\frac{\pi t^2}{2}}.
    \end{align*}
    As a result, \begin{equation}\label{eq:ris1}
        ( W_{I}\phi\ast W_I\phi^*)(t,s)=2^{-d} e^{-\frac{\pi(t^2+s^2)}{2}}.
    \end{equation}
 An easy computation gives the  Wigner distribution of the Gaussian $\phi$:\begin{equation*}
     W\phi(x,\xi)= 2^{\frac{d}{2}} e^{-2\pi (x^2+\xi^2)}.
 \end{equation*}
The convolution $W\phi\ast W\phi^*=W\phi\ast W\phi$ is given by \begin{align}
    (W\phi\ast W\phi^*)(t,s)=&2^d\int_{\rdd} e^{-2\pi (x^2+\xi^2)}e^{-2\pi((x-t)^2+(\xi-s)^2)}dxd\xi\notag\\
    =& 2^d e^{-2\pi (t^2+s^2)}\int_{\rdd}e^{-\pi (4x^2+4\xi^2 -2xt-2\xi s)}dxd\xi\notag\\=& 2^d e^{-2\pi (t^2+s^2)}\int_{\rd} e^{-\pi(4x^2-4xt)}dx\int_{\rd}e^{-\pi(4\xi^2-4\xi s)}d\xi\notag\\
    =& 2^{-d} e^{-2\pi (t^2+s^2)} e^{\pi (t^2+s^2)}\notag\\
    =& 2^{-d}e^{-\pi(t^2+s^2)}\label{eq:01}.
\end{align}
Since \eqref{eq:ris1}$\not=\eqref{eq:01}$  we obtain the claim.
\end{proof}

\begin{example}\label{counterex2}
 Consider  $\A_{ST}\in Sp(2d,\bR)$ in \eqref{AST}. The related metaplectic Wigner distribution $W_{\A_{ST}}$ is the STFT, see \eqref{eq:stft-metap}. As in the previous example, we choose  $f=g=\varphi_1=\varphi_2=\phi$, with $\phi$ defined in \eqref{eq:phi}.
    Then \begin{equation}\label{eq:ex2}
    	 V_{\phi}\phi\ast V_{\phi}\phi^*\not=W\phi\ast W\phi^*.
    \end{equation}
    
    \end{example}
 \begin{proof} 
    An easy computation (see, e.g., \cite{book}) shows:
    \begin{equation*}
        V_{\phi}\phi(x,\xi)=2^{-\frac{d}{2}}e^{-\frac{\pi}{2}(x^2+\xi^2)}e^{-\pi ix\xi}.
    \end{equation*}
    Now we compute $V_{\phi}\phi\ast V_{\phi}\phi^*$:
    \begin{align*}
        (V_{\phi}\phi\ast V_{\phi}\phi^*)(t,s)=& \int_{\rdd} 2^{-d}e^{-\frac{\pi}{2}(x^2+\xi^2)}e^{-\pi i x\xi} e^{-\frac{\pi}{2}((x-t)^2+(\xi-s)^2)}e^{\pi i (x-t)(\xi-s)}dxd\xi\\ 
        =& 2^{-d} e^{-\frac{\pi}{2}(t^2+s^2)} e^{\pi i ts}\int_{\rdd} e ^{-\pi (x^2+\xi^2-xt-\xi s)} e ^{-\pi i t\xi} e ^{-\pi i xs}dx d\xi \\
        =& 2^{-d}  e^{-\frac{\pi}{2}(t^2+s^2)} e^{\pi i ts}\int_{\rd} e ^{-\pi (x^2-xt)} e ^{-\pi i xs}dx \int_{\rd}  e ^{-\pi (\xi^2-\xi s)} e ^{-\pi i \xi t}d\xi \\ 
        =& 2^{-d}  e^{-\frac{\pi}{2}(t^2+s^2)} e^{\pi i ts} e^{\frac{\pi}{4}(t^2+s^2)}\int_{\rd} e ^{-\pi (x-\frac{t}{2})^2} e ^{-2\pi i x\frac{s}{2}}dx \\
        &\qquad\qquad\times\quad \int_{\rd}e ^{-\pi (\xi-\frac{s}{2})^2} e ^{-2\pi i \xi(\frac{t}{2})}d\xi \\
        =& 2^{-d} e^{-\frac{\pi}{4}(t^2+s^2)} e^{\pi i ts} \cF (T_{t/2}\phi)(s/2)\cF (T_{s/2}\phi)(t/2)\\
        =& 2^{-d} e^{-\frac{\pi}{4}(t^2+s^2)} e^{\pi i ts} M_{-t/2}\hat \phi(s/2) M_{-s/2}\hat \phi (t/2)\\
        =& 2^{-d} e^{-\frac{\pi}{2}(t^2+s^2)}.
    \end{align*}
    The obtained expression is clearly different from \eqref{eq:01}.
    \end{proof}

\begin{remark}
    From the two previous counterexamples we can easily build an explicit counterexample showing that the equality \eqref{eq:a-loc} is  false if $W_{\cA}$ is not covariant. For instance, one can consider an $\cA$-localization operators related to any of the metaplectic Wigner distributions in the previous examples, Gaussian windows $\phi$ and the symbol $a=\delta_{z_0}\in\cS'(\rdd)$, with $z_0\in\rdd$  a point where the convolution products \eqref{eq:counteres1} or \eqref{eq:ex2}  are different. 
\end{remark}

Examples \ref{counter1} and \ref{counterex2} highlight what we already expect from Lemma \ref{viceversa3.3}. The following example shows that, even if we change the symbol and the windows, it is not generally possible to write an arbitrary $\cA$-localization operator in the classical form.

\begin{example}\label{ex3}
	 Consider  $J\in Sp(2d,\bR)$ so that $\hat{J}=\cF$,  the symbol $a\equiv 1\in \cS'(\rdd)$, and $\varphi_1=\varphi_2=\phi$, with $\phi$ defined in \eqref{eq:phi}.
	Then,  there exist no symbol $b\in\cS'(\rdd)$ and no pair of windows $\phi_1,\phi_2\in\cS(\rd)$ such that \begin{equation}\label{eq:locdiffsimbol}
		A_{a,\cA}^{\varphi_1,\varphi_2}=A_{b}^{\phi_1,\phi_2}. 
	\end{equation}
\end{example}
	\begin{proof}
By contradiction. Assume \eqref{eq:locdiffsimbol} holds true. 	Then the connection \eqref{AlocopElena} gives
\begin{equation*}
		\operatorname{Op}_{\cA}(a\ast W_{\cA}(\varphi_2, \varphi_1))=\operatorname{Op}_w(b\ast W(\phi_2,\phi_1)).
	\end{equation*}
	By Lemma \ref{lemmaComm}, this is equivalent to asking \begin{equation*}
		b\ast W(\phi_2,\phi_1)=\hat \cA_{1/2}\hat \cA^{-1}(a\ast W_{\cA}(\varphi_2,\varphi_1)).
	\end{equation*}
	The term $b\ast W(\phi_2,\phi_1)$ is a convolution between a tempered distribution and a Schwartz function, so it is a regular distribution associated to a slowly increasing, $\cC^{\i}$ function on $\rdd$. Analyzing the right-hand side we get: \begin{align*}
		\hat \cA_{1/2}\hat \cA^{-1}(a\ast W_{\cA}(\varphi_2,\varphi_1))&=\hat \cA_{1/2} \cF^{-1}(a\ast \cF(\varphi_2\otimes\overline{\varphi_1}))\\
		&= \hat \cA_{1/2}(\cF^{-1}(a)\cdot (\varphi_2\otimes\overline{\varphi_1})).
	\end{align*}
	Now, $\cF^{-1}(a)=\cF^{-1}(1)=\delta_{2d}$, which is the Dirac delta distribution on $\rdd$. Since $\varphi_1,\varphi_2$ are standard gaussians, we have that $\overline{(\varphi_2\otimes\overline{\varphi_1})(0,0)}=1$, so, $\delta_{2d}\cdot (\varphi_2\otimes\overline{\varphi_1})=\delta_{2d}=\delta_d\otimes \overline{\delta_d}$. Therefore,
	\begin{equation*}
		\hat \cA_{1/2}(\cF^{-1}(a)\cdot (\varphi_2\otimes\overline{\varphi_1}))=\hat \cA_{1/2}(\delta_d\otimes \overline{\delta_d})=W(\delta_d)= \delta_d\otimes 1,
	\end{equation*}
	which is a contradiction, since $\delta_d\otimes 1$ is not a regular distribution. 
	\end{proof}
	\subsection{ Schwartz kernel of $A_{a,\cA}^{\varphi_1,\varphi_2}$ }
In what follows we compute the Schwartz kernel of the $\cA$-localization operators.
\begin{proposition}\label{propkernelAloc}
    Let $\cA\in Sp(2d,\bR)$, $\varphi_1,\varphi_2\in\cS(\rd)$ and $a\in\cS'(\rdd)$, then the Schwartz kernel $k$ of $A_{a,\cA}^{\varphi_1,\varphi_2}$ is given by:\begin{equation}\label{kernelofAloc}
        k=\hat A^{-1}(a\ast W_{\cA}(\varphi_2,\varphi_1)).
    \end{equation}
\end{proposition}
\begin{proof}
    By definition of $A_{a,\cA}^{\varphi_1,\varphi_2}$ we have\begin{equation*}
        A_{a,\cA}^{\varphi_1,\varphi_2}=\operatorname{Op}_{\cA}(a\ast W_{\cA}(\varphi_2,\varphi_2).
    \end{equation*}
    By applying Corollary \ref{cor:kernel} we get the thesis.
\end{proof}

 \subsubsection{Totally-Wigner decomposable $\cA\in Sp(2d,\bR)$} This class of metaplectic Wigner distributions was introduced in \cite[Definition 4.1]{CGR2022} and refers to symplectic matrices of the type
 \begin{equation}\label{SympDec}
 	\mathcal{A}=\mathcal{A}_{FT2}\mathcal{D}_E,
 \end{equation}
 where  $\mathcal{D}_E$ is defined in \eqref{defDLVC}.
 \begin{definition}\label{WigDecDef} We say that $\mathcal{A}\in Sp(2d,\mathbb{R})$ is a \textbf{totally Wigner-decomposable} (symplectic) matrix if (\ref{SympDec}) holds for some $E\in GL(2d,\mathbb{R})$. If $\cA$ is totally Wigner-decomposable, we say that $W_\cA$ is of the \textbf{classic type}.
 \end{definition}

 They have been largely studied in the literature \cite{Bayer, CorderoTrapasso, toft}, see the recent survey \cite{Giacchi}.
 
 In what follows we infer an explicit formula for the kernel of the related  $\cA$- localization operator.

 \begin{proposition}\label{propkerneltwd}
 Let  $\mathcal{A}\in Sp(2d,\mathbb{R})$  be  \textbf{totally Wigner-decomposable}. Suppose that $E$ and $E^{-1}$ have block decomposition:\begin{equation}
         E=\begin{pmatrix}A&B\\C&D\end{pmatrix},\quad E^{-1}=\begin{pmatrix}
             A'&B'\\ C'&D'
         \end{pmatrix}.
     \end{equation}
     Then, for $a\in\cS'(\rd)$, $\f_1,\f_2\in\cS(\rd)$, the Schwartz kernel $k$ of $\cA^{\varphi_1,\varphi_2}_{a,\cA}$ is given by:\begin{align}
         &k(x,y)=\notag\\&\int_{\rd}T_{(t,0)}\cF_2^{-1}a(E^{-1}(x,y))\varphi_2(At+B(C'x+ D' y)) \overline{\varphi_1(Ct+ D (C'x+D'y))}dt.\label{kerneltwd}
     \end{align}
     Where the integral is to be understood in the weak sense.
 \end{proposition}
 \begin{proof}
 By Proposition \ref{propkernelAloc} we can write \begin{equation*}
     k=\hat \cA^{-1}(a\ast W_{\cA}(\varphi_2,\varphi_1))=\hat \cD_{E^{-1}}\cF_2^{-1}(a\ast W_{\cA}(\varphi_2,\varphi_1)).
 \end{equation*}
      We recall that \begin{equation*}
         \cF_2^{-1}(T\ast\phi)(x,y)=\int_{\rd}\cF_2^{-1}T(x-t,y)\cF_2^{-1}\phi(t,y)dt,\quad  T\in\cS'(\rdd), \phi\in\cS(\rdd).
     \end{equation*}
     Then \begin{align*}
     \cF_2^{-1}(a\ast W_{\cA}(\varphi_2,\varphi_1))(x,y)&=\int_{\rd}\cF_2^{-1}a(x-t,y)\cF_2^{-1}W_{\cA}(\varphi_2,\varphi_1)(t,y)dt\\
     =& \int_{\rd}\cF_2^{-1}a(x-t,y)\cF_2^{-1}\cF_2\hat \cD_E(\varphi_2\otimes \overline{\varphi_1})(t,y)dt\\
     =&|\det E|^{\frac{1}{2}}\int_{\rd}\cF_2^{-1}a(x-t,y)(\varphi_2\otimes \overline{\varphi_1})(E(t,y))dt\\
     =&|\det E|^{\frac{1}{2}}\int_{\rd}T_{(t,0)}\cF_2^{-1}a(x,y)\varphi_2(At+B y) \overline{\varphi_1(C t+ D y)}dt.
     \end{align*}
     So, \begin{align*}
          \hat\cD_{E^{-1}}&\cF_2^{-1}(a\ast W_{\cA}(\varphi_2,\varphi_1))(x,y)\\ =&\int_{\rd}T_{(t,0)}\cF_2^{-1}a(E^{-1}(x,y))\varphi_2(At+B(C'x+ D' y)) \overline{\varphi_1(Ct+ D (C'x+D'y))}dt.
     \end{align*}
    This is the desired expression.
 \end{proof}
\subsection{Continuity properties}
 In order to study the continuity properties of a $\cA$ localization operator $A_{a,\cA}^{\varphi_1,\varphi_2}$, it is useful to compute its Weyl symbol. 
 \begin{proposition}\label{propWeylAloc}
      Let $\cA\in Sp(2d,\bR)$, $\varphi_1,\varphi_2\in\cS(\rd)$ and $a\in\cS'(\rdd)$. The Weyl symbol of  $A_{a,\cA}^{\varphi_1,\varphi_2}$ is given by:\begin{equation}\label{WeylAloc}
         \sigma=\hat\cA_{1/2}\hat \cA^{-1}(a\ast W_{\cA}(\varphi_2,\varphi_1))\in \cS'(\rdd).
     \end{equation}
 \end{proposition}
 \begin{proof}
    Consider the $\cA$-pseudodifferential operator representing the  $\cA$-localization operator \eqref{AlocopElena}:\begin{equation*}
         A_{a,\cA}^{\varphi_1,\varphi_2}=\operatorname{Op}_{\cA}(a\ast W_{\cA}(\varphi_2,\varphi_1)).
     \end{equation*}
     By Lemma \ref{lemmaComm} \begin{equation*}
         \operatorname{Op}_{\cA}(a\ast W_{\cA}(\varphi_2,\varphi_1))= \operatorname{Op}_w(b),
     \end{equation*}
     where $b=\hat\cA_{1/2}\hat \cA^{-1}(a\ast W_{\cA}(\varphi_2,\varphi_1))$, which gives \eqref{WeylAloc}.
 \end{proof}
                
    
The following results are  generalizations of Theorem 6 in \cite{CorderoQB} to $\cA$-localization operators. For the sake of clarity, we distinguish the two cases $1\leq p\leq 2$ and $2<p\leq \i$.
\begin{theorem}\label{schattencont1}
    Let $1\leq p\leq 2$, $a\in M^{p,\i}(\rdd)$, $\varphi_1,\varphi_2\in M^{1}(\rd)$ and $\cA\in Sp(2d,\bR)$. Then $A_{a,\cA}^{\varphi_1,\varphi_2}\in\cS_p$ and we have the estimate:\begin{equation}\label{schattenbound1}
        \norm{A_{a,\cA}^{\varphi_1,\varphi_2}}_{\cS_p}\lesssim \norm{a}_{M^{p,\i}}\norm{\varphi_1}_{M^1}\norm{\varphi_2}_{M^1}.
    \end{equation}

\end{theorem}
\begin{proof}
    By Proposition \ref{propWeylAloc}, the Weyl symbol of $A_{a,\cA}^{\varphi_1,\varphi_2}$ takes the form:\begin{equation}\label{eq:WeylE}
        \sigma=\hat\cA_{1/2}\hat \cA^{-1}(a\ast W_{\cA}(\varphi_2,\varphi_1)).
    \end{equation}
Since $\varphi_1,\varphi_2\in M^1(\rd)$, then $\varphi_2\otimes \overline{\varphi_1}\in M^1(\rdd)$ and by the continuity of metaplectic operators on modulation spaces $M^p$, cf., \cite{CauliNicolaTabacco2019}, we have that $W_{\cA}(\varphi_1,\varphi_2)\in M^1(\rdd)$ and \begin{equation*}
    \norm{W_{\cA}(\varphi_2,\varphi_1)}_{M^1(\rdd)}\lesssim \norm{\varphi_2\otimes \overline{\varphi_1}}_{M^1(\rdd)}\lesssim \norm{\varphi_1}_{M^1(\rd)}\norm{\varphi_2}_{M^1(\rd)}.
\end{equation*}
So, $a\ast W_{\cA}(\varphi_2,\varphi_1)$ is a convolution between an element of $M^{p,\i}(\rdd)$ and $M^1(\rdd)$, respectively. By the convolution properties for modulation spaces (see, e.g., Proposition 2.4. in \cite{corderogrochenig}) we have that $a\ast W_{\cA}(\varphi_2,\varphi_1)\in M^{p,1}(\rdd)$ with the norm estimate:\begin{equation*}
   \norm{ a\ast W_{\cA}(\varphi_2,\varphi_1)}_{M^{p,1}}\lesssim \norm{a}_{M^{p,\i}}\norm{W_{\cA}(\varphi_2,\varphi_1)}_{M^p}\lesssim \norm{a}_{M^{p,\i}}\norm{\varphi_1}_{M^1}\norm{\varphi_2}_{M^1}.
\end{equation*}
The continuous inclusion $M^{p,1} (\rdd)\hookrightarrow M^p(\rdd)$ (see Section 2 above) and the continuity of metaplectic operators on $M^p(\rdd)$ give the estimate:\begin{align*}
    \norm{\sigma}_{M^p}\lesssim \norm{a\ast W_{\cA}(\varphi_2,\varphi_1)}_{M^p}\lesssim \norm{a\ast W_{\cA}(\varphi_2,\varphi_1)}_{M^{p,1}}.
\end{align*}

Since the Weyl symbol $\sigma$ is in $M^p(\rdd)$, Theorem \ref{teoweylsymb} infers that the operator $A_{a,\cA}^{\varphi_1,\varphi_2}$ is in $\cS_p$ and it satisfies the norm estimate \eqref{schattenbound1}.
\end{proof}

We now treat the case $p> 2$.
\begin{theorem}\label{schattencont2}
    Let $2< p\leq \i$, $a\in M^{p,\i}(\rdd)$, $\varphi_1,\varphi_2\in M^{1}(\rd)$ and $\cA\in Sp(2d,\bR)$ with block decomposition \eqref{blockA}. If $\cA$ satisfies the block conditions:\begin{equation}
       \begin{cases}\label{blockconditions}
            A_{31}+A_{32}= 0_{d\times d}\\
            A_{41}+A_{42}= 0_{d\times d}\\
            A_{34}-A_{33}= 0_{d\times d}\\
            A_{43}+A_{44}= 0_{d\times d},
        \end{cases}\end{equation}
    then $A_{a,\cA}^{\varphi_1,\varphi_2}\in\cS_p$ and we have the estimate:\begin{equation}\label{schattenbound12}
        \norm{A_{a,\cA}^{\varphi_1,\varphi_2}}_{\cS_p}\lesssim \norm{a}_{M^{p,\i}}\norm{\varphi_1}_{M^1}\norm{\varphi_2}_{M^1}.
    \end{equation}

\end{theorem} 

\begin{proof} Our goal is to show that the Weyl symbol $\sigma$ in \eqref{eq:WeylE} is in $M^{p,p'}(\rdd)$. Then,   Theorem \ref{teoweylsymb} allows to conclude. Given $2\leq p\leq \i$, then $1\leq p'\leq p$, and we have the continuous embedding $M^{p,1}(\rdd)\hookrightarrow M^{p,p'}(\rdd)$. Hence, the same argument as in the proof of Theorem \ref{schattencont1} gives \begin{equation*}
    \norm{ a\ast W_{\cA}(\varphi_2,\varphi_1)}_{M^{p,p'}}\lesssim \norm{a}_{M^{p,\i}}\norm{\varphi_1}_{M^1}\norm{\varphi_2}_{M^1}.
\end{equation*}
If $p \not=p'$, by the characterization presented by  Führ and  Shafkulovska in \cite[Theorem 3.2]{Fuhr}, the metaplectic operator $\hat \cA_{1/2}\hat\cA^{-1}$ is everywhere defined and continuous from $M^{p,p'}(\rdd)$ to itself, if and only if the projection $\cA_{1/2}\cA^{-1}$ is upper block triangular. To conclude the proof, we verify that the conditions in \eqref{blockconditions} are equivalent to state that $\cA_{1/2}\cA^{-1}$ is upper block triangular. If $\cA$ has block decomposition \eqref{blockA}, then, by \eqref{inverseA}, \begin{equation*}
    \cA^{-1}=\begin{pmatrix}
        A_{33}&A_{43}&-A_{13}&-A_{23}\\
        A_{34}&A_{44}&-A_{14}&-A_{24}\\
        -A_{31}&-A_{41}&A_{11}&A_{21}\\
        -A_{32}&-A_{42}&A_{12}&A_{22}
    \end{pmatrix}.
\end{equation*}
For $\tau=1/2$ in \eqref{Atau} we obtain \begin{equation*}
    \cA_{1/2}=\begin{pmatrix}\frac{1}{2}I_{d\times d} & \frac{1}{2}I_{d\times d} & 0_{d\times d} & 0_{d\times d}\\
		0_{d\times d} & 0_{d\times d} & \frac{1}{2}I_{d\times d} & -\frac{1}{2}I_{d\times d}\\
		0_{d\times d} & 0_{d\times d} & I_{d\times d} & I_{d\times d}\\
		-I_{d\times d} & I_{d\times d} & 0_{d\times d} & 0_{d\times d}\end{pmatrix}.
\end{equation*}
By computing the matrix multiplication $\cA_{1/2}\cA^{-1}$, it is easy to find that the $2d\times2d$ left-lower block is given by the matrix:\begin{equation*}
   \begin{pmatrix}
       -A_{31}-A_{32} & -A_{41}-A_{42}\\
       A_{34}-A_{33}& A_{43}+A_{44}
   \end{pmatrix},
\end{equation*}
which is $0_{2d\times2d}$ if and only if \eqref{blockconditions} holds.
\end{proof}

\begin{remark}
    From Proposition \ref{propcovariant} it is evident that the conditions \eqref{blockconditions} hold for every covariant metaplectic Wigner distribution, hence, we retrieve the results for classical localization operators in Theorems \ref{schattencont1} and \ref{schattencont2}.
\end{remark}

The following example shows that if  $\cA$ does not satisfies conditions \eqref{blockconditions}, then the operator $A_{a,\cA}^{\varphi_1,\varphi_2}$ may not be bounded on $L^2(\rd)$.

\begin{example}
    Let $\varphi_1=\varphi_2=\phi,$ for $\phi$ defined in \eqref{eq:phi}. Consider the Fourier operator $\hat J=\cF$ and the symbol  $a\equiv 1\in M^{\i}(\rdd)$. Observe that the symplectic matrix $J$ does not satisfies conditions \eqref{blockconditions}. In fact,  we have $$A_{31}+A_{32}=-I_{d\times d}\ne 0_{d\times d}.$$
    Then $ A_{a,\cA}^{\varphi_1,\varphi_2}$ is not bounded on $\lrd$.
\end{example}
\begin{proof}
	To extend $A_{a,\cA}^{\varphi_1,\varphi_2}:\cS(\rd)\to \cS'(\rd)$ to a bounded, linear operator on $L^2(\rd)$, it is necessary (and sufficient) that:\begin{equation*}
		\sup_{\norm{g}_2=1;\ g\in \cS(\rd)} |\la A_{a,\cA}^{\varphi_1,\varphi_2}f,g\ra|<\i, \quad \forall f\in \cS(\rd).
	\end{equation*}
Fix $f\in \cS(\rd)$, then, for every $g\in \cS(\rd)$ with $\norm{g}_2=1$ we have,
	\begin{align*}
		|\la A_{a,\cA}^{\varphi_1,\varphi_2}f,g\ra|&=|\la a, W_{\cA}(g,f)\ast W_{\cA}(\varphi_2,\varphi_1)^*\ra|\\
		&=|\la a, \cF(g\otimes \overline{f})\ast \cF\big(\cF^{-1}(W_{\cA}(\varphi_2,\varphi_1)^*)\big)\ra|\\
		&=|\la a, \cF\bigg((g\otimes \overline{f})\cdot \big(\cF^{-1}(W_{\cA}(\varphi_2,\varphi_1)^*)\big)\bigg)\ra|\\
		&=|\la \cF^{-1}a, (g\otimes \overline{f})\cdot \big(\cF^{-1}(W_{\cA}(\varphi_2,\varphi_1)^*)\big)\ra|.
	\end{align*}
Since, $\cF^{-1}a=\delta$, the Dirac delta distribution centered in $(0,0)\in\rdd$ and $$W_{\cA}(\varphi_2,\varphi_1)^*(\eta)=\overline{\cF(\varphi_2\otimes \overline{\varphi_1})(-\eta)}=\overline{\cF^{-1}(\varphi_2\otimes \overline{\varphi_1})(\eta)},\ \forall \eta \in \rdd,$$
we can write\begin{align*}
		|\la A_{a,\cA}^{\varphi_1,\varphi_2}f,g\ra|=|\overline{g(0)}f(0) \overline{\bigg(\cF^{-1}\big(\overline{\cF^{-1}(\varphi_2\otimes \overline{\varphi_1})}\big)\bigg)(0,0)}|.
	\end{align*}
	Since $\varphi_1,\varphi_2=\phi$, it follows that  $\overline{\bigg(\cF^{-1}\big(\overline{\cF^{-1}(\varphi_2\otimes \overline{\varphi_1})}\big)\bigg)(0,0)}=1 $. In conclusion, \begin{equation*}
		\sup_{\norm{g}_2=1;\ g\in \cS(\rd)} |\la A_{a,\cA}^{\varphi_1,\varphi_2}f,g\ra|=|f(0)|\sup_{\norm{g}_2=1;\ g\in \cS(\rd)} |g(0)|,
	\end{equation*}
	which is not finite for if  $f(0)\ne 0$. Take, for instance, $\frac1{\eps^{d/2}}g(\frac{t}{\eps})$, with $\|g\|_2=1$ and $g(0)\not=0$.
\end{proof}

Under the hypotheses of Theorems \ref{schattencont1},  the operator $A_{a,\cA}^{\varphi_1,\varphi_2}$ is a bounded linear operator on $L^2(\rd)$. We report here the calculation for its adjoint.

\begin{proposition}\label{adjointprop}
    Assume that $A_{a,\cA}^{\varphi_1,\varphi_2}$ is a continuous mapping on $L^2(\rd)$, then its adjoint operator is given by\begin{equation}\label{adjoint}
        (A_{a,\cA}^{\varphi_1,\varphi_2})^*=A_{\overline{a}, \cA \cD_S}^{\varphi_2,\varphi_1},
    \end{equation}
where \begin{equation}\label{eq:S}
	S=\begin{pmatrix}
        0_{d\times d}& I_{d\times d}\\
        I_{d\times d}& 0_{d\times d},
    \end{pmatrix},
\end{equation}
and the related symplectic matrix $\mathcal{D}_S$ is defined in \eqref{defDLVC}.
\end{proposition}
\begin{proof}
    Let $f,g\in L^2(\rd)$, then \begin{align*}
        \la A_{a,\cA}^{\varphi_1,\varphi_2}f,g\ra=&\la a, W_{\cA}(g,f)\ast W_{\cA}(\varphi_2,\varphi_1)^*\ra\\
        =&\la a, \hat\cA(g\otimes \overline{f})\ast(\hat \cA(\varphi_2\otimes \overline{\varphi_1}))^*\ra.
    \end{align*}
    If $S$ is given by \eqref{eq:S}, the related symplectic matrix is $\mathcal{D}_S$, defined in \eqref{defDLVC}, and the  metaplectic operator  $\widehat{\mathcal{D}}_S=\mathfrak{T}_S$ in \eqref{es22} $(ii)$ switches the two variables:
    \begin{equation*}
      \mathfrak{T}_S F(x,y)=F(y,x),\quad \forall F\in L^2(\rdd),\ x,y\in \rd.
    \end{equation*}
    Therefore, \begin{align*}
        \la a, \hat\cA(g\otimes \overline{f})\ast(\hat \cA(\varphi_2\otimes \overline{\varphi_1}))^*\ra&= \la a, \hat\cA\widehat\cD_S(\overline{f}\otimes g)\ast(\hat \cA\widehat \cD_S(\overline{\varphi_1}\otimes \varphi_2))^*\ra\\
        &=\la a, \overline{\hat\cA\widehat \cD_S(f\otimes \overline{g})\ast(\hat \cA\widehat \cD_S(\varphi_1\otimes \overline{\varphi_2}))^*}\ra\\
        &=\overline{\la \overline{a}, \hat\cA\widehat \cD_S(f\otimes \overline{g})\ast(\hat \cA\widehat \cD_S(\varphi_1\otimes \overline{\varphi_2}))^*\ra}\\
        &=\overline{\la \overline{a}, W_{\cA D_S}(f,g)\ast W_{\cA \cD_S}(\varphi_1,\varphi_2)^*\ra}\\
        &=\overline{\la A_{\overline{a}, \cA \cD_S}^{\varphi_2,\varphi_1}g,f\ra}\\
        &=\la f, A_{\overline{a}, \cA \cD_S}^{\varphi_2,\varphi_1}g\ra,
    \end{align*}
    which concludes the proof.
\end{proof}

\begin{theorem}\label{continuityonMp}
    Let $1< p\leq \i$ $\cA\in Sp(2d,\bR)$, $\varphi_1,\varphi_2\in M^1(\rd)$ and $a\in M^{r,\i}(\rdd)$, where $r=\min\{p,p'\}$. Then $A_{a,\cA}^{\varphi_1,\varphi_2}$ is bounded from $M^{p}(\rd)$ to itself.
\end{theorem}
\begin{proof}

Given $h\in \cS'(\rd)$, we recall that:\begin{equation*}
    \norm{h}_{M^p}=\sup_{\norm{g}_{M^{p'}}=1}|\la h,g\ra|, \quad \forall 1<p\leq \i.
\end{equation*}
So, for every $f\in \cS(\rd)$,\begin{align*}
    \norm{A_{a,\cA}^{\varphi_1,\varphi_2}f}_{M^p}=&\sup_{\norm{g}_{M^{p'}}=1}|\la A_{a,\cA}^{\varphi_1,\varphi_2}f ,g\ra|\\
    =&\sup_{\norm{g}_{M^{p'}}=1}|\la a, W_{\cA}(g,f)\ast W_{\cA}(\varphi_2,\varphi_1)^*\ra|.
\end{align*}
Since $f\in M^{p}(\rd),g\in M^{p'}(\rd)$ and $\varphi_1,\varphi_2\in M^1(\rd)$, it follows that $W_{\cA}(g,f)\in M^{\max\{p,p'\}}(\rdd), W_{\cA}(\varphi_2,\varphi_1)^*\in M^1(\rdd)$. We use the convolution properties  for modulation spaces (Proposition 2.4 \cite{corderogrochenig}) to infer$$W_{\cA}(g,f)\ast W_{\cA}(\varphi_2,\varphi_1)^*\in M^{\max \{p,p'\},1}(\rdd)$$ and \begin{equation*}
   \norm{W_{\cA}(g,f)\ast W_{\cA}(\varphi_2,\varphi_1)^*}_{M^{\max\{p,p'\},1}}\lesssim \norm{W_{\cA}(g,f)}_{M^{\max\{p,p'\}}}\norm{W_{\cA}(\varphi_2,\varphi_1)^*}_{M^1}.
\end{equation*} Moreover, by H\"older inequality, \begin{equation*}
    |\la a, W_{\cA}(g,f)\ast W_{\cA}(\varphi_2,\varphi_1)^*\ra|\lesssim  \norm{a}_{M^{\min\{p,p'\},\i}}\norm{W_{\cA}(g,f)\ast W_{\cA}(\varphi_2,\varphi_1)^*}_{M^{\max\{p,p'\},1}}.
\end{equation*}
Therefore, since $a, \varphi_1,\varphi_2$ are fixed, \begin{align*}
    \norm{A_{a,\cA}^{\varphi_1,\varphi_2}f}_{M^p}\lesssim \sup_{\norm{g}_{M^{p'}}=1} \norm{W_{\cA}(g,f)}_{M^{\max\{p',p\}}}.
\end{align*}
The continuity of $\hat \cA$ on $M^{\max\{p',p\}}(\rdd)$ and the embedding $M^{\max\{p',p\}}(\rd)\hookrightarrow M^k(\rd)$, with $k=p$ or $p'$, implies: \begin{align*}
    \sup_{\norm{g}_{M^{p'}}=1} \norm{W_{\cA}(g,f)}_{M^{\max\{p',p\}}} &\lesssim \sup_{\norm{g}_{M^{p'}}=1}\norm{g\otimes \overline{f}}_{M^{\max\{p',p\}}}\\
    &\lesssim \sup_{\norm{g}_{M^{p'}}=1}\norm{g}_{M^{p'}}\norm{f}_{M^p}\\
    &\lesssim \norm{f}_{M^p},
\end{align*}

showing the boundedness of $A_{a,\cA}^{\varphi_1,\varphi_2}$ on $M^p(\rd)$.
\end{proof}

\begin{remark}
  (i)  By applying the same strategy of the proof above, one can easily show that, if $1\leq p<\i$ and $a\in M^{p',\i}(\rdd)$, then $A_{a,\cA}^{\varphi_1,\varphi_2}$ is bounded from $M^{p}(\rd)$ to $M^{p'}(\rd)$. \par
  (ii) Similar arguments can be used to show the continuity properties of $\cA$-localization operators on weighted modulation spaces, we leave the details to the interested reader.
\end{remark}

\section*{Acknowledgements}
The   authors have been supported by the Gruppo Nazionale per l’Analisi Matematica, la Probabilità e le loro Applicazioni (GNAMPA) of the Istituto Nazionale di Alta Matematica (INdAM).

\end{document}